\documentclass[a4paper,12pt]{amsart}

\pdfoutput=1
\usepackage{setspace}
\usepackage[left=28mm,head=30mm,bottom=30mm,foot=20mm,right=28mm]{geometry}
\usepackage{amsmath,amsthm,amsfonts,amssymb,mathtools,mathrsfs,url,bm,enumitem,stackengine}
\usepackage{newtxtext}
\usepackage[varvw]{newtxmath}
\usepackage{graphicx}
\usepackage{tikz}
\usepackage{tikz,amsmath,amssymb}
 
\newcommand{\axesframe}{%
  \draw (0,0) rectangle (2,2);
  \draw[->] (0,0)--(2.2,0) node[right]{$\alpha$};
  \draw[->] (0,0)--(0,2.2) node[above]{$\beta$};
  \foreach \t/\l in {1/1, {5/3}/\frac53, 2/2} {
    \draw (\t,0)--(\t,-0.04) node[below,font=\scriptsize]{$\l$};
    \draw (0,\t)--(-0.04,\t) node[left,font=\scriptsize]{$\l$};
  }
}

\newcommand{\guides}{%
  \draw[dashed,gray] (1,0)--(1,2);
  \draw[dashed,gray] ({5/3},0)--({5/3},2);
  \draw[dashed,gray] (0,{5/3})--(2,{5/3});
  \draw[dashed,gray] (0,1.5)--(2,0.5);
  \draw[dashed,gray] (1,2)--(2,1.5);
}

\allowdisplaybreaks
\usepackage{hyperref}
\hypersetup{
	colorlinks=true,
	linkcolor=blue,
	citecolor=red,
}
\usepackage{appendix}

\mathtoolsset{showonlyrefs=true}
\theoremstyle{definition}

\newtheorem{remark}{Remark}[section]

\theoremstyle{plain}

\newtheorem{theorem}{Theorem}[section]

\newtheorem{corollary}{Corollary}[section]

\newtheorem{lemma}{Lemma}[section]

\renewcommand{\phi}{\varphi}
\let\phi\varphi
\let\epsilon\varepsilon

\def\R{\mathbb R}

\def\norm#1{\Vert#1\Vert}

\def\vu{ u }

\begin{document}

	\title[Liouville Theorems for Fractional MHD]{Liouville Theorems for the Stationary Fractional Magnetohydrodynamic Equations}

\author{Albert Ai}
\address{Beijing International Center for Mathematical Research, Peking University, No. 5 Yiheyuan Road Haidian District, Beijing 100871, People’s Republic of China}
\email{albertai@bicmr.pku.edu.cn}		
	
\author{Gast\'{o}n Vergara-Hermosilla}
\address{School of Mathematics, Harbin Institute of Technology, Xidazhi Street, Harbin 150001,
Heilongjiang, People’s Republic of China}
\email{gaston.v-h@outlook.com (corresponding author)}	

	\maketitle

  \begin{abstract}
 We establish Liouville-type theorems for the three-dimensional stationary fractional magnetohydrodynamic (MHD) system with different orders of dissipation.
 Our first result shows that every smooth solution satisfying the natural homogeneous Sobolev energy conditions is trivial whenever $1\leq\alpha\leq 5/3$ and $3\leq\alpha+2\beta\leq5$, 
 without any additional integrability assumptions on the velocity or magnetic fields.
  This extends previous results to new parameter regimes, including cases with $\beta<1$. Beyond this region, we prove a complementary Liouville theorem under suitable integrability conditions only on the radial components of the velocity and magnetic fields. The proofs rely on the Caffarelli--Silvestre extension and localized energy estimates.
 \end{abstract}

 \section{Introduction and main results}
The Liouville problem for the stationary Navier--Stokes equations in $\mathbb R^3$ asks whether a smooth solution must vanish identically under suitable assumptions on its behavior at infinity. A classical version of this problem concerns solutions with finite Dirichlet integral,
\[
\int_{\mathbb R^3}|\nabla u|^2\,dx<\infty,
\]
and remains open in its full generality.
An important contribution was made by Galdi, who proved that the additional assumption $u\in L^{9/2}(\mathbb R^3)$ is sufficient to guarantee triviality. Since then, a number of Liouville-type criteria have been established under different integrability and regularity assumptions; see, for instance, \cite{chae14,CW,Kozonoetal,Seregin16,Sereginetwang,Tan25,V2026}
 and the references therein.

A natural extension of this problem is obtained by replacing the Laplacian with a fractional power $(-\Delta)^{\alpha/2}$. 
This type of dissipation, which goes back to the work of J.-L.~Lions \cite{Lions}, allows one to study stationary flows with different orders of dissipation. 
For the stationary fractional Navier--Stokes equations, Liouville-type results have been obtained in several settings; see, among others, 
\cite{wx18,yang22,J-me,LiuZuo,EVH2026,LeeLee26,LiuTan26}. The relation between the dissipation exponent and the integrability of the 
velocity field is particularly relevant here. Indeed, the homogeneous Sobolev embedding
\[
\dot H^{\alpha/2}(\mathbb R^3)
\hookrightarrow L^{\frac{6}{3-\alpha}}
(\mathbb R^3
)
\]
shows that the exponent $9/2$ in Galdi's criterion corresponds precisely to $\alpha=5/3$. This observation provides a natural reference point for understanding the range of fractional dissipation considered in this paper.

Our main interest is the stationary magnetohydrodynamic (MHD) system, which describes the interaction between the velocity field of an 
electrically conducting fluid and its magnetic field.
 In this system, the velocity $u$ and the magnetic field $b$ are coupled through the Lorentz force and the induction equation, while $\Pi$ denotes the total pressure. 
 For the classical stationary MHD equations, several Liouville-type theorem are available in the literature. See, for instance, \cite{chaeweng,chaewolfMHD,chaekimwolf,schulz,chaeJMP}. 
The fractional setting introduces an additional difficulty in that the velocity and magnetic fields may have different dissipation orders, denoted by 
$\alpha$ and $\beta$, respectively. 
As a consequence, the two fields need not have the same regularity, and 
the interaction between their dissipative 
effects becomes an essential part of the analysis.

To the best of our knowledge, relatively few works have addressed Liouville-type theorems for the stationary fractional MHD system. 
Among the relevant contributions, we mention the results of Zeng \cite{zeng} and Lee--Lee \cite{LeeLee26}. 
These works provide important criteria for the triviality of solutions, but leave room to investigate parameter ranges in which the two dissipation orders interact more directly. 
Thus, our purpose is to establish such criteria under the natural energy assumptions
\[
u\in\dot H^{\alpha/2}(\mathbb R^3),
\qquad
b\in\dot H^{\beta/2}(\mathbb R^3),
\]
without imposing additional integrability assumptions on either field in the parameter ranges covered by our first theorem.

More precisely, we consider the stationary fractional MHD system in $\mathbb R^3$:
\begin{equation}\label{SMHD}
\begin{cases}
(-\Delta)^{\alpha/2}u+(u\cdot\nabla)u+\nabla\Pi
=(b\cdot\nabla)b,\\[2mm]
(-\Delta)^{\beta/2}b+(u\cdot\nabla)b
=(b\cdot\nabla)u,\\[2mm]
\nabla\cdot u=0,\qquad \nabla\cdot b=0,
\end{cases}
\qquad 0<\alpha,\beta<2,
\end{equation}
with the uniform decay conditions
\[
u(x) \to 0, 
\quad 
b(x) \to 0, 
\quad 
\text{as } 
|x|\to \infty
.
\]

Here, the fractional Laplacians are defined through the Fourier multipliers with symbols $|\xi|^\alpha$ and $|\xi|^\beta$. The system includes the stationary fractional Navier--Stokes equations when $b\equiv0$, while the choice $\alpha=\beta=2$ formally recovers the classical stationary MHD system.

Our first result identifies a region in the 
$(\alpha,\beta)$-plane
 where the energy assumptions alone imply that both fields vanish. 
 To continue, we introduce the set  
\[
\mathcal A:=\Bigl\{\alpha,\beta \in (0, 2) \ :\ 1\le\alpha\le\tfrac53,\ 3\le\alpha+2\beta\le5 \Bigr\}.
\]
The significance of this region is that the admissible dissipation orders are not restricted to zones of the plane where $\beta\ge 1$. Instead, the mixed interaction between the velocity and magnetic fields leads to a more flexible condition  which exploits the combination $\alpha+2\beta$.
\begin{theorem}\label{thm1}
Let $(u,b)$ be a smooth solution of \eqref{SMHD} such that
\[
u\in\dot H^{\alpha/2}(\mathbb R^3),
\qquad
b\in\dot H^{\beta/2}(\mathbb R^3),
\]
with $(\alpha,\beta)\in\mathcal A$. Then
\[
u\equiv0,\qquad b\equiv0.
\]
\end{theorem}

\begin{remark}\label{rem:range}
The parameter region in Theorem~\ref{thm1} strictly contains the parameter regions obtained by Zeng \cite{zeng} and Lee--Lee \cite{LeeLee26},which are given, respectively, by
\[
\mathcal{B}_{\text{Zeng}} = \Bigl\{(\alpha,\beta)\ :\ 
1\le\alpha,\beta
\le
\tfrac{5}{3}
\Bigr\}
\]
and
\[
\mathcal{B}_{\text{Lee-Lee}} = 
\Bigl\{(\alpha,\beta)\ :\ 1\le\alpha<\tfrac{5}{3},\quad \tfrac{5}{3}<\beta<2,\quad \alpha+2\beta<5\Bigr\}.
\]
In particular, our result covers points on the line $\alpha
+
2\beta=5$ with $1\leq\alpha<5/3$,  as well as a triangular region in which $\beta<1$. In this latter range,
the Sobolev embedding for the magnetic field gives an exponent below $3$,
so the theorem applies even when the magnetic field does not belong to the Lebesgue spaces required by the previous criteria. The key point is that we retain the mixed terms involving $u$ and $b$ in the energy-flux estimates,
rather than estimating them separately by purely cubic terms. This preserves the interaction between the two dissipation orders and allows us to work under the single balance condition $3\leq\alpha+2\beta\leq5$,
without requiring an independent lower bound on $\beta$.
In the following figure we sketch these sets alongside $\mathcal A$ to provide the reader with a graphical comparison.

\begin{figure}[htbp]
\centering
\begin{minipage}{0.325\textwidth}\centering
\begin{tikzpicture}[scale=1.75]
\fill[blue!15] (1,1)--({5/3},{2/3})--({5/3},{5/3})--(1,2)--cycle;
\axesframe
\guides
\draw (1,1)--({5/3},{2/3})--({5/3},{5/3})--(1,2)--cycle;
\node at (1.34,1.4) {$\mathcal A$};
\end{tikzpicture}
\end{minipage}
\begin{minipage}{0.325\textwidth}\centering
\begin{tikzpicture}[scale=1.75]
\fill[blue!15] (1,{5/3})--({5/3},{5/3})--(1,2)--cycle;
\axesframe
\guides
\draw (1,{5/3})--(1,2);
\draw[dashed] (1,{5/3})--({5/3},{5/3})--(1,2);
\node[font=\small] at (1.5,1.35) {$\mathcal B_{\text{Lee-Lee}}$};
\end{tikzpicture}
\end{minipage}
\begin{minipage}{0.325\textwidth}\centering
\begin{tikzpicture}[scale=1.75]
\fill[blue!15] (1,1) rectangle ({5/3},{5/3});
\axesframe
\guides
\draw (1,1) rectangle ({5/3},{5/3});
\node at (1.33,1.33) {$\mathcal B_{\text{Zeng}}$};
\end{tikzpicture}
\end{minipage}
\caption{The sets $\mathcal A$, $\mathcal B_{\text{Lee-Lee}}$ and $\mathcal B_{\text{Zeng}}$ in the common window $(0,2)^2$. Solid edges belong to the set, while black dashed edges do not. 
The gray dashed lines are the reference lines $\alpha=1$, $\alpha=\frac53$, $\beta=\frac53$, $\alpha+2\beta=3$ and $\alpha+2\beta=5$, common to the three panels.}
\end{figure}
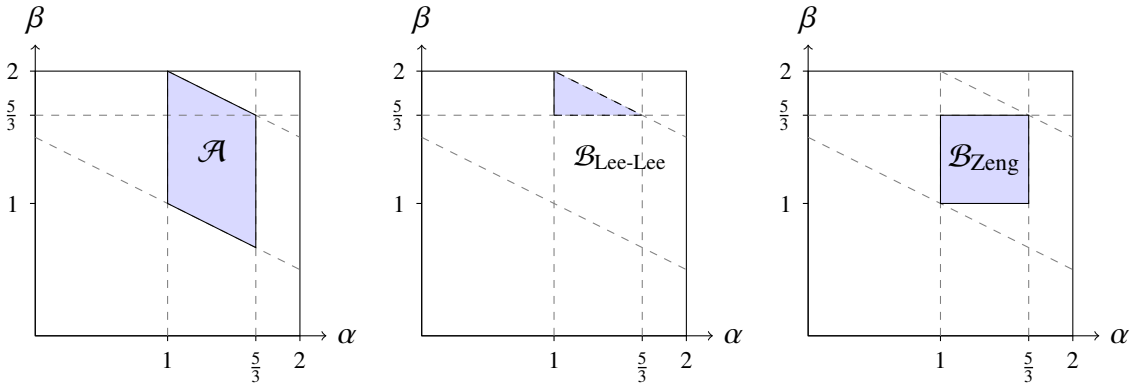
\end{remark}

An immediate consequence of Theorem~\ref{thm1} is a Liouville theorem for the stationary fractional Navier--Stokes equations. In fact, by setting $b\equiv0$, the MHD system reduces to the fractional Navier--Stokes system, and the magnetic dissipation exponent $\beta$ can be chosen freely, provided that $(\alpha,\beta)\in\mathcal A$. This observation allows us to obtain the following result throughout the range $1\leq\alpha\leq5/3$, without any additional integrability assumption on the velocity field. To the best of our knowledge, Zeng was the first to establish this conclusion in the same range; see \cite[Corollary 1.5]{zeng}. Our proof, however, follows a different approach.

\begin{corollary}\label{cor:NS}
Let $1\leq\alpha\leq\frac53$, and let
$u\in\dot H^{\frac\alpha2}(\mathbb R^3)$ be a smooth solution of
\[
(-\Delta)^{\frac\alpha2}u
+(u\cdot\nabla)u+\nabla\Pi=0,
\qquad
\nabla\cdot u=0
\quad\text{in }\mathbb R^3.
\]
Then $u\equiv0$.
\end{corollary}

We next consider the parameter ranges outside $\mathcal A$.
 In this setting, we show that triviality can  still be obtained
  by imposing suitable integrability  assumptions on the radial  components of the velocity 
  and magnetic fields, 
\[
u_\rho:=u\cdot\frac{x}{|x|},
\qquad
b_\rho:=b\cdot\frac{x}{|x|}.
\]
This approach is motivated by recent Liouville-type theorems for the classical and fractional stationary Navier--Stokes equations 
\cite{EVH2026,V2026b}.
In contrast to conditions imposed  on individual Cartesian components, assumptions on  radial components are invariant under rotations about the origin. Moreover, 
they do not require additional integrability of the tangential components of either field.

Our second result extends this perspective to the stationary fractional MHD system. It covers the parameter region outside $\mathcal A$, 
except for a region where the requirements arising from the velocity and magnetic fluxes  are incompatible within our method. 
In particular, our theorem  distinguishes the cases in which the radial integrability assumptions are 
needed for the velocity field,  for the magnetic field,  or for both. 
We present the result below.
 
\begin{theorem}\label{thm.radial}
Let $0<\alpha,\beta<2$, and let $(u,b)$ be a smooth solution of \eqref{SMHD} satisfying
$
u\in\dot H^{\alpha/2}(\mathbb R^3)$,  $b\in\dot H^{\beta/2}(\mathbb R^3)$.
Assume that $u_\rho\in L^{r_1}(\R^3)$ and $b_\rho\in L^{r_2}(\R^3)$ for $r_1,r_2\in[1,\infty]$ such that
\begin{align}
&\text{(R1)}\quad r_1\ge\frac{3}{\alpha}\ \text{ if }\alpha<1,\qquad r_1\le\frac{3}{\alpha-1}\ \text{ if }\alpha>\frac53,\label{R1}\\
&\text{(R2)}\quad r_1\ge\frac{3}{\beta}\ \text{ if }\alpha+2\beta<3,\qquad r_1\le\frac{3}{\beta-1}\ \text{ if }\alpha+2\beta>5,\label{R2}\\
&\text{(R3)}\quad r_2\ge\frac{6}{\alpha+\beta}\ \text{ if }\alpha+2\beta<3,\qquad r_2\le\frac{6}{\alpha+\beta-2}\ \text{ if }\alpha+2\beta>5.\label{R3}
\end{align}
Then $u\equiv 0$ and $b\equiv 0$.
\end{theorem}

Some remarks are in order.
\begin{remark}
We stress that Theorem \ref{thm.radial} requires additional
integrability assumptions only on the radial components of the
velocity and magnetic fields, namely $u_\rho$ and $b_\rho$,
without imposing any further integrability conditions on their
tangential components of $u$ and of $b$. 
In particular, 
the criterion depends only
on these two components, while the remaining components are
constrained only by the Sobolev regularity assumed in the statement of the theorem.
\end{remark}
\begin{remark}
For comparison, Zeng's Liouville theorem for the stationary fractional MHD system in \cite{zeng}  is based on the assumption that
\[
(u_j,b_j)\in L^{\vec p_j}(\mathbb R^3)
\times L^{\vec q_j}(\mathbb R^3),
\qquad j=1,2,3,
\]
where $\vec p_j=(p_{j,1},p_{j,2},p_{j,3})$ and
$\vec q_j=(q_{j,1},q_{j,2},q_{j,3})$, with
$p_{j,l},q_{j,l}\in[3,\infty)$ satisfying
\[
\sum_{l=1}^3\frac{1}{p_{j,l}}\geq\frac{2}{3},
\qquad
\sum_{l=1}^3\frac{1}{q_{j,l}}\geq\frac{2}{3},
\qquad j=1,2,3.
\]
Thus, Zeng's result imposes anisotropic integrability conditions on each Cartesian component of both fields, combining different integrability requirements along the coordinate directions to ensure the triviality of the solution.
By contrast, Theorem \ref{thm.radial} assumes only
\[
u_\rho\in L^{r_1}(\mathbb R^3),
\qquad
b_\rho\in L^{r_2}(\mathbb R^3),
\]
for exponents $r_1,r_2\in[1,\infty]$ satisfying
\emph{(R1)}--\emph{(R3)}. These restrictions depend on the
fractional orders $\alpha$ and $\beta$ and allow us to formulate
the Liouville criterion entirely in terms of the radial
components, rather than imposing additional integrability
conditions on each Cartesian component of the two fields.
\end{remark}
\begin{remark}\label{rem:cases}
The conditions (R1)--(R3) are active only in some parts of $(0,2)^2$. Thus, writing them, one finds the following list, where the Roman numerals refer to Figure \ref{fig:regions}.
\begin{center}
\renewcommand{\arraystretch}{1.35}
\begin{tabular}{@{}llll@{}}
 & region & condition on $u_\rho$ & condition on $b_\rho$\\
\hline
$\mathcal A$ & $1\le\alpha\le\frac53,\ 3\le\alpha+2\beta\le5$ & none & none\\
I & $\alpha<1,\ \alpha+2\beta<3$ & $r_1\ge\frac{3}{\min\{\alpha,\beta\}}$ & $r_2\ge\frac{6}{\alpha+\beta}$\\
II & $\alpha<1,\ \alpha+2\beta\ge3$ & $r_1\ge\frac{3}{\alpha}$ & none\\
III & $1\le\alpha\le\frac53,\ \alpha+2\beta<3$ & $r_1\ge\frac{3}{\beta}$ & $r_2\ge\frac{6}{\alpha+\beta}$\\
IV & $1\le\alpha\le\frac53,\ \alpha+2\beta>5$ & $r_1\le\frac{3}{\beta-1}$ & $r_2\le\frac{6}{\alpha+\beta-2}$\\
V & $\alpha>\frac53,\ 3\le\alpha+2\beta\le5$ & $r_1\le\frac{3}{\alpha-1}$ & none\\
VI & $\alpha>\frac53,\ \alpha+2\beta>5$ & $r_1\le\frac{3}{\max\{\alpha,\beta\}-1}$ & $r_2\le\frac{6}{\alpha+\beta-2}$\\
$\mathcal G$ & $\alpha>\frac53,\ \alpha+2\beta<3$ & incompatible requirements & ---
\end{tabular}
\end{center}
In particular, on $\mathcal A$, Theorem \ref{thm.radial} contains no radial hypothesis and reduces to Theorem \ref{thm1}, whereas on each of the regions I--VI it provides a result with a one-sided condition.
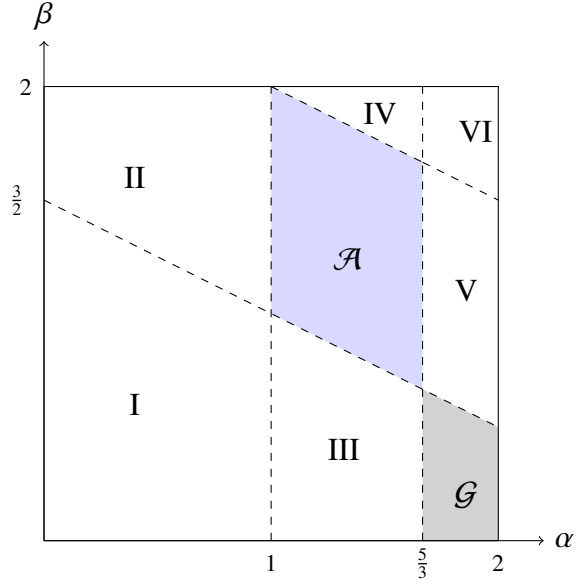
\begin{figure}[htbp]
\centering
\begin{tikzpicture}[scale=3]
\fill[blue!15] (1,1)--({5/3},{2/3})--({5/3},{5/3})--(1,2)--cycle;
\fill[gray!35] ({5/3},0)--(2,0)--(2,0.5)--({5/3},{2/3})--cycle;
\draw (0,0) rectangle (2,2);
\draw[dashed] (1,0)--(1,2);
\draw[dashed] ({5/3},0)--({5/3},2);
\draw[dashed] (0,1.5)--(2,0.5);
\draw[dashed] (1,2)--({5/3},{5/3});
\draw[dashed] ({5/3},{5/3})--(2,1.5);
\draw[->] (0,0)--(2.2,0) node[right]{$\alpha$};
\draw[->] (0,0)--(0,2.2) node[above]{$\beta$};
\node[font=\scriptsize] at (1,-0.1) {$1$};
\node[font=\scriptsize] at ({5/3},-0.1) {$\frac53$};
\node[font=\scriptsize] at (2,-0.1) {$2$};
\node[font=\scriptsize] at (-0.08,2) {$2$};
\node[font=\scriptsize] at (-0.12,1.5) {$\frac32$};
\node at (1.34,1.24) {$\mathcal A$};
\node at (0.4,0.6) {I};
\node at (0.4,1.6) {II};
\node at (1.32,0.4) {III};
\node at (1.48,1.88) {IV};
\node at (1.86,1.1) {V};
\node at (1.9,1.8) {VI};
\node at (1.86,0.2) {$\mathcal G$};
\end{tikzpicture}
\caption{The region $\mathcal A$ (shaded in blue), where no radial hypothesis is needed, and the regions I--VI and $\mathcal G$ of Remark \ref{rem:cases} in the square $(0,2)^2$. The dashed lines are $\alpha=1$, $\alpha=\frac53$, $\alpha+2\beta=3$ (from $(0,\frac32)$ to $(2,\frac12)$) and $\alpha+2\beta=5$.}
\label{fig:regions}
\end{figure}
\end{remark}
\begin{remark}
\emph{The region $\mathcal G$.} For $\alpha>\frac53$ and $\alpha+2\beta<3$ (so that $\beta<\frac23$) the requirements coming from the terms $|u|^2u_\rho$ and $|b|^2u_\rho$ push $r_1$ in opposite directions. The former asks for $r_1\le\frac{3}{\alpha-1}$ and the latter for $r_1\ge\frac{3}{\beta}$, and $\frac{3}{\beta}>\frac{3}{\alpha-1}$ because $\alpha-\beta>1$ there.  
\end{remark}
\begin{remark}\emph{The case $b\equiv0$.} If $b\equiv0$, the second equation in \eqref{SMHD} is trivially satisfied and the first one is the stationary fractional Navier--Stokes system with dissipation $(-\Delta)^{\alpha/2}$. Only the terms $|u|^2u_\rho$ and $\Pi_uu_\rho$ appear in the flux, so only (R1) is used. It requires $r_1\ge\frac{3}{\alpha}$ if $\alpha<1$, nothing if $1\le\alpha\le\frac53$, and $r_1\le\frac{3}{\alpha-1}$ if $\frac53<\alpha<2$. For $\frac53<\alpha<2$, the range treated in \cite{EVH2026}, this contains the condition $\frac{3}{\alpha}\le r_1\le\frac{3}{\alpha-1}$ of that paper.
\end{remark}

The paper is organized as follows. In Section~\ref{sect:prel}, we recall the Caffarelli--Silvestre extension and its main properties. Sections~\ref{proofs} 
and~\ref{sect:radial} contain the proofs of Theorems~\ref{thm1} and~\ref{thm.radial}, respectively.\\

Given $R>1$, in what follows, and throughout this paper, we use the notation
\[
\mathcal{C}_R := \{ x \in \mathbb{R}^3 \ : \ R < |x| < 2R \}.
\]
We let $\psi\in \mathcal{C}^\infty_c(\mathbb{R}^3)$ be a smooth radial non-increasing cut-off
function such that $0\le \psi\le 1$, $\psi(x)=1$ for $|x|\le 1$, and
$\psi(x)=0$ for $|x|\ge 2$. Given $R>1$, set
$  \psi_R(x)=\psi\!\left(\frac{x}{R}\right)$.

We also consider a non-increasing function
$\chi\in \mathcal{C}^\infty([0,\infty))$ such that $\chi(y)=1$ for
$y\leq 1$ and $\chi(y)=0$ for $y\geq 2$. Given $R>1$, we define
$\chi_R(y) = \chi\!\left(\frac{y}{R}\right)$. Note that $|\nabla\psi_R|\le CR^{-1}$, $|\Delta\psi_R|\le CR^{-2}$, $|\chi_R'|\le CR^{-1}$ and $|\chi_R''|\le CR^{-2}$, for some absolute constant $C>0$.

\section{Preliminaries}\label{sect:prel}

Since the operators $(-\Delta)^{\alpha/2}$ and $(-\Delta)^{\beta/2}$ are nonlocal, some of the classical arguments used to prove Liouville type theorems for the stationary Navier--Stokes and MHD equations cannot be carried over directly. Instead, we make use of the Caffarelli--Silvestre extension, which allows one to represent a fractional Laplacian on $\mathbb{R}^{3}$ as the boundary trace of a local problem involving the operator $\bar{\Delta}$ in the half-space $\mathbb{R}^{4}_{+}$. In the following, we shall write $\bar{\nabla}$, $\overline{\operatorname{div}}$ and $\bar{\Delta}$ for the differential operators associated with the extended variables in $\mathbb{R}^{4}_{+}$. Since the two equations in \eqref{SMHD} involve different fractional Laplacians, we will need two extensions, one for $u$ and one for $b$, with different weights. For this reason we state the result for a generic exponent $\sigma$.

\begin{theorem}\cite{cs07}\label{thm:CS}
Given $\sigma \in (0,2)$,
consider $v \in \dot{H}^{\frac \sigma 2}(\mathbb{R}^3)$, $\mu = 1 - \sigma$,
and the weighted homogeneous energy space
$\dot H^1(\mathbb R^4_+,y^\mu)$, obtained as the completion of
$C_c^\infty(\overline{\mathbb{R}^4_+})$ with respect to the seminorm
\[
\|w\|_{\dot H^1(\mathbb R^4_+,y^\mu)}
:=
\left(
\int_{\mathbb R^4_+}
y^\mu |\bar\nabla w|^2\, d  x\, d  y
\right)^{1/2}.
\]
 Then $v$ admits a unique extension $v^* \in \dot H^1(\mathbb{R}^4_+, y^\mu)$ satisfying
\[
\bar{\Delta}_\mu v^*(x, y) := \bar{\Delta} v^* + \frac{\mu}{y} \partial_y v^* = \frac{1}{y^\mu} \overline{\operatorname{div}}(y^\mu \bar{\nabla} v^*) = 0  ,
\]
and
\[
v^*(x, 0) = v(x).
\]
In addition, there exists $C_\sigma>0$, depending only on $\sigma$,  such that:
\begin{enumerate}
    \item the operator $(-\Delta)^{\frac \sigma 2}  v$ is given by
    \[
    (-\Delta)^{\frac \sigma 2}  v(x) = - \lim_{y \to 0} C_\sigma ( y^\mu \partial_y v^*(x, y)),
    \]
    \item and, the energy equality
    \[
    \int_{\mathbb{R}^3} |(-\Delta)^{\frac \sigma 4}  v|^2 dx  = C_\sigma \int_{\mathbb{R}^4_+} y^\mu |\bar{\nabla} v^*|^2 dx dy.
    \]
\end{enumerate}
Moreover, $v^*(\cdot,y)=P^{(\sigma)}_y*v$, where $P_y^{(\sigma)}(x)=c_{3,\sigma}\,y^{\sigma}(|x|^2+y^2)^{-\frac{3+\sigma}{2}}$ is the fractional Poisson kernel and $c_{3,\sigma}>0$ is chosen so that $\|P_y^{(\sigma)}\|_{L^1(\mathbb R^3)}=1$.
\end{theorem}

In what follows we write
\[
\mu:=1-\alpha, \qquad \nu:=1-\beta, \qquad p:=\frac{6}{3-\alpha}, \qquad q:=\frac{6}{3-\beta},
\]
so that, by the Sobolev embedding, $u\in L^{p}(\R^3)$ and $b\in L^{q}(\R^3)$. We denote by $u^*$ and $b^*$ the Caffarelli--Silvestre extensions of $u$ and $b$, associated with the weights $y^{\mu}$ and $y^{\nu}$ respectively, and by $C_\alpha$, $C_\beta$ the corresponding constants given by Theorem \ref{thm:CS}. Note that $\mu,\nu>-1$ because $\alpha,\beta<2$.

\section{Proof of Theorem \ref{thm1}}\label{proofs}

We begin with two lemmas that will be used in the proof of Theorem \ref{thm1}. 
The first is stated for a generic exponent, so that it can be applied both to $u$ and to $b$. 
Its proof is analogous to that of \cite{EVH2026}; we include it here for completeness. 
The proof of the second is postponed to the end of this section.

\begin{lemma}\label{lemma.with.E}
Let $\sigma\in(0,2)$, $v\in \dot H^{\sigma/2}(\mathbb R^3)$ and let $v^*$ be its Caffarelli--Silvestre extension. For every $R>1$ let $E_R\subseteq B_{2R}$ be a measurable set. Then
\[
\frac1{R^2}\int_0^{2R}\int_{E_R} y^{1-\sigma}|v^*|^2 dxdy\xrightarrow[R\to\infty]{}0
.
\]
\end{lemma}

\begin{lemma}\label{lemma I R}
Let $(\alpha,\beta)\in\mathcal A$, $u\in\dot H^{\alpha/2}(\R^3)$, $b\in\dot H^{\beta/2}(\R^3)$. Let $\Pi$ be the pressure defined in \eqref{eq:pressure} below.
 Then
\[
\lim_{R\to +\infty}
\int_{ C_R }\bigl(|\vu|^2+|b|^2\bigr)
| \vu |
| \nabla  \psi_R |
 dx
+
 \int_{ C_R }
 | \Pi  |
 | \vu  |
|  \nabla  \psi_R |
  dx
 = 0.
\]
\end{lemma}

\begin{proof}[Proof of Theorem \ref{thm1}]

To begin, note that, since $u\in \dot H^{\alpha/2}(\mathbb{R}^3)$ and $b\in \dot H^{\beta/2}(\mathbb{R}^3)$, the Sobolev
embedding gives $u\in L^{p}(\mathbb{R}^3)$ and $b\in L^{q}(\mathbb{R}^3)$. Since $\alpha,\beta>0$ we have $p,q>2$, and therefore $u_iu_j\in L^{p/2}(\mathbb{R}^3)$ and $b_ib_j\in L^{q/2}(\mathbb{R}^3)$ with $1<\frac p2,\frac q2<+\infty$. Then, taking the divergence of the momentum equation and using
$\nabla\cdot u=\nabla\cdot b=0$, we obtain the elliptic equation
\begin{equation}\label{ellipeq}
-\Delta \Pi=\partial^i\partial^j(u_iu_j-b_ib_j)=:f .
\end{equation}
Now, since \eqref{ellipeq} determines $\Pi$ only up to an additive harmonic function, we fix the
pressure through the Calder\'on--Zygmund representation
\begin{equation}\label{eq:pressure}
\Pi:=\Pi_1+\Pi_2, \qquad \Pi_1:=R^iR^j(u_iu_j), \qquad \Pi_2:=-R^iR^j(b_ib_j),
\end{equation}
where $R^i$ denotes the $i$-th Riesz transform on $\mathbb{R}^3$. By the Calder\'on--Zygmund
theorem,
\begin{equation}\label{eq:CZ}
\|\Pi_1\|_{L^{\frac p2}(\mathbb{R}^3)}\le C\,\|u\|_{L^{p}(\mathbb{R}^3)}^2<+\infty,
\qquad
\|\Pi_2\|_{L^{\frac q2}(\mathbb{R}^3)}\le C\,\|b\|_{L^{q}(\mathbb{R}^3)}^2<+\infty .
\end{equation}
We claim that this is the only solution   of \eqref{ellipeq} in $L^{\frac p2}(\mathbb{R}^3)+L^{\frac q2}(\mathbb{R}^3)$. Indeed, if
$\tilde \Pi$ is another solution in this space, then $h:=\Pi-\tilde \Pi=h_1+h_2$ is harmonic on
$\mathbb{R}^3$, with $h_1\in L^{\frac p2}(\mathbb{R}^3)$ and $h_2\in L^{\frac q2}(\mathbb{R}^3)$. Thus, considering the mean value property and H\"older's
inequality, for every $x\in\mathbb{R}^3$ and $\rho>0$,
\[
|h(x)|\le \frac{1}{|B_\rho|}\int_{B_\rho(x)}(|h_1|+|h_2|)\,dy
\le |B_\rho|^{-\frac{2}{p}}\,\|h_1\|_{L^{\frac p2}(\mathbb{R}^3)}
+|B_\rho|^{-\frac{2}{q}}\,\|h_2\|_{L^{\frac q2}(\mathbb{R}^3)}\xrightarrow[\rho\to\infty]{}0 ,
\]
thus, we get $h\equiv 0$. Throughout the paper, $\Pi$ always denotes this normalized pressure.
On the other hand,  since  $u$ and $b$ are smooth, by using elliptic regularity theory on \eqref{ellipeq}, we conclude that the pressure term $\Pi$ is also smooth.

Then, testing the first equation of \eqref{SMHD} with $\psi_R u $ and the second one with $\psi_R b$,  using $\nabla\cdot u=\nabla\cdot b=0$ and integrating by parts, we get
\begin{multline}\label{eq2-energy}
\int_{\mathbb{R}^3} \psi_R u \cdot (-\Delta)^{\frac \alpha 2}  u \, dx
+
\int_{\mathbb{R}^3} \psi_R b \cdot (-\Delta)^{\frac \beta 2}  b \, dx
= 
\\
\int_{\mathbb{R}^3} \left(\frac {|u|^2} 2 +\frac {|b|^2} 2 + \Pi \right)u \cdot \nabla \psi_R \, dx
-\int_{\mathbb{R}^3} (u\cdot b)\,(b \cdot \nabla \psi_R) \, dx.
\end{multline}

We next treat the fractional dissipation terms on the left-hand side by means of the Caffarelli--Silvestre extension. For this purpose, we consider $\Phi(x,y):=\psi_R(x)\chi_R(y)$ and the identity
(see \cite[Eq. (3.2)]{yang22}), applied to each component of $u$ with $\sigma=\alpha$ and to each component of $b$ with $\sigma=\beta$,
\begin{align}
\int_{\mathbb{R}^3} \psi_R u \cdot (-\Delta)^{\frac \alpha 2}  u \, dx
& =
 C_\alpha \int_{\mathbb{R}^4_+} y^\mu |\overline\nabla u^*|^2 \Phi \, dx \, dy
 -
  C_\alpha \int_{\mathbb{R}^4_+} \frac{|u^*|^2}{2}
   \overline
 {\text{div}}
  \bigl( y^\mu \overline\nabla \Phi \bigr) \, dx \, dy
  ,\label{eq25-u}\\
\int_{\mathbb{R}^3} \psi_R b \cdot (-\Delta)^{\frac \beta 2}  b \, dx
& =
 C_\beta \int_{\mathbb{R}^4_+} y^\nu |\overline\nabla b^*|^2 \Phi \, dx \, dy
 -
  C_\beta \int_{\mathbb{R}^4_+} \frac{|b^*|^2}{2}
   \overline
 {\text{div}}
  \bigl( y^\nu \overline\nabla \Phi \bigr) \, dx \, dy
  .\label{eq25-b}
\end{align}
Then plugging \eqref{eq25-u} and \eqref{eq25-b} into \eqref{eq2-energy}, we write
\begin{multline}\label{eq3-energy}
 C_\alpha\int_{\mathbb{R}^4_+} y^\mu |\overline \nabla u^*|^2 \Phi \, dx \, dy
 +
 C_\beta\int_{\mathbb{R}^4_+} y^\nu |\overline \nabla b^*|^2 \Phi \, dx \, dy
\\=
  \frac{C_\alpha}{2} \int_{\mathbb{R}^4_+}  |u^*|^2   \overline
 {\text{div}} \bigl( y^\mu  \overline\nabla \Phi \bigr) \, dx \, dy
 +
  \frac{C_\beta}{2} \int_{\mathbb{R}^4_+}  |b^*|^2   \overline
 {\text{div}} \bigl( y^\nu  \overline\nabla \Phi \bigr) \, dx \, dy
\\
+
 \int_{\mathbb{R}^3} \left(\frac {|u|^2} 2 +\frac {|b|^2} 2 + \Pi \right)u \cdot \nabla \psi_R \, dx
-\int_{\mathbb{R}^3} (u\cdot b)\,(b \cdot \nabla \psi_R) \, dx.
\end{multline}

Now, considering the supports of $\chi_R,\ \psi_R$ and the facts that $\overline\nabla \Phi(x,y) = ( \chi_R \nabla\psi_R, \psi_R \chi'_R)$ and, for a generic weight $y^{\kappa}$, we can write 
\[ \overline
 {\text{div}} ( y^\kappa \overline\nabla \Phi ) = y^\kappa \chi_R\Delta \psi_R + \psi_R\bigl(\kappa y^{\kappa-1}\chi_R'+y^\kappa\chi_R''\bigr),\] 
 we notice that the first term is supported in $\mathcal C_R\times[0,2R]$, while the second one is supported in $B_{2R}\times[R,2R]$, where $y^{\kappa-1}\le R^{-1}y^{\kappa}$. We deduce the estimate
 
\[
\Bigl|\int_{\mathbb{R}^4_+}  |v^*|^2  \overline
 {\text{div}} \bigl( y^\kappa \overline\nabla \Phi(x,y)  \bigr) dx  dy\Bigr|
\leq
\frac C {R^2}
 \int_{0}^{2R}   \!  \!   \!
 \int_{ \mathcal C_R} y^\kappa |v^*|^2 dxdy
 +
 \frac C {R^2}
 \int_{R}^{2R} \!  \!  \!     \int_{ B_{2R}} y^\kappa |v^*|^2 dx dy ,
\]
where $C$ depends on $\kappa$ but not on $R$. Take $(v,\kappa)=(u,\mu)$ and $(v,\kappa)=(b,\nu)$. Then, applying Lemma \ref{lemma.with.E} (with $\sigma=\alpha$, respectively $\sigma=\beta$) with $E_R=\mathcal C_R$, we conclude that the first term on the right-hand side tends to zero as $R\to\infty$. On the other hand, since $y^\kappa|v^*|^2 \geq 0$, we can write
\[
 \int_{R}^{2R} \!  \!  \!     \int_{ B_{2R}} y^\kappa |v^*|^2 dx dy
\le
 \int_{0}^{2R} \!  \!  \!     \int_{ B_{2R}} y^\kappa |v^*|^2 dx dy,
 \]
 therefore, considering Lemma \ref{lemma.with.E} with $E_R=B_{2R}$ we obtain that the second term tends to zero also. Considering the same approach for $b^*$, we conclude
 \begin{equation}
 \lim_{R\to\infty}
 \int_{\mathbb{R}^4_+}  |u^*|^2 \overline{\text{div}} \bigl( y^\mu \overline\nabla \Phi(x,y)  \bigr) dx  dy
 =
 \lim_{R\to\infty}
 \int_{\mathbb{R}^4_+}  |b^*|^2 \overline{\text{div}} \bigl( y^\nu \overline\nabla \Phi(x,y)  \bigr) dx  dy
 =0.
 \end{equation}

Now, note that, since $\psi_R$ is radial and $\nabla \psi_R$ is supported in $\mathcal C_R$, we can write
\begin{multline}
\Bigl|\int_{\mathbb{R}^3} \left(\frac {|u|^2} 2 +\frac {|b|^2} 2 + \Pi \right)u \cdot \nabla \psi_R \, dx
-\int_{\mathbb{R}^3} (u\cdot b)\,(b \cdot \nabla \psi_R) \, dx\Bigr|
\\ \leq
C
\int_{ \mathcal C_R }
\bigl(|\vu|^2+|b|^2\bigr)
| \vu |
| \nabla   \psi_R|
 dx
+
C
 \int_{ \mathcal C_R }
 | \Pi  |
 | \vu |
|  \nabla \psi_R |
  dx
  ,
\end{multline}
where we used $|(u\cdot b)(b\cdot\nabla\psi_R)|\le|u||b|^2|\nabla\psi_R|$. By Lemma \ref{lemma I R}, the right-hand side tends to zero as $R\to\infty$.

Thus, provided with these results,  we take the limit as $R\to \infty$ in \eqref{eq3-energy}. Since $\psi$ is radial and non-increasing, and $\chi$ is non-increasing, $\Phi=\Phi_R$ increases pointwise to $1$ as $R\to\infty$, and the monotone convergence theorem applied to the left-hand side of \eqref{eq3-energy} yields
\[
C_\alpha\int_{\mathbb{R}^4_+} y^\mu |\overline{\nabla} u^*|^2 \, dx \, dy
+
C_\beta\int_{\mathbb{R}^4_+} y^\nu |\overline{\nabla} b^*|^2 \, dx \, dy = 0.
\]
Thus, since $C_\alpha,C_\beta>0$, the both integrals above vanish. 
Hence, we conclude that  $u^*$ and $b^*$ are constant in $\mathbb R^4_+$, and so are their traces $u(x)=u^*(x,0)$ and $b(x)=b^*(x,0)$. 
Then, considering that $u\in L^{p}(\mathbb R^3)$ and $b\in L^{q}(\mathbb R^3)$, and the fact that the only constant function belonging to
 $ L^{r}(\mathbb R^3) $, $r<\infty$, is the zero function, we conclude that  $u\equiv 0$ and $b\equiv 0$. This finishes the proof.
\end{proof}

\begin{proof}[Proof of Lemma \ref{lemma.with.E}]
In the following we denote by $q_\sigma$ the Sobolev exponent, {\it i.e.} $q_\sigma:=\frac{6}{3-\sigma}>2$, and we write $\kappa:=1-\sigma\in(-1,1)$.

By Theorem \ref{thm:CS}, the Caffarelli--Silvestre extension can be written as
\[
v^*(x,y)=P_y^{(\sigma)}*v(x).
\]
By Young's inequality,
\begin{equation}\label{eq3.6}
\|v^*(\cdot,y)\|_{L^{q_\sigma}(\mathbb R^3)}
\leq
\|v\|_{L^{q_\sigma}(\mathbb R^3)}.
\end{equation}
Now, considering H\"older inequality with  $\frac 1 {q_\sigma/2} + \frac 1{q_\sigma/(q_\sigma-2)}=1$, we get
\begin{align}
\int_{E_R}|v^*(x,y)|^2\,dx
& =
\int_{\mathbb{R}^3} \big(|v^*(x,y)|^2\mathbf1_{E_R}(x)\big)\,dx
 \\
 &
 \le \||v^*(\cdot,y)|^2 \|_{L^{q_\sigma/2}}\,\|\mathbf1_{E_R}\|_{L^{q_\sigma/(q_\sigma-2)}}
 \\
 & =\norm{v^*(\cdot,y)}_{L^{q_\sigma}}^2\,|E_R|^{1- \frac 2 {q_\sigma} }.
\end{align}
Then, since $|E_R|^{1-2/q_\sigma}\le|B_{2R}|^{1-2/q_\sigma}=CR^{3(1-2/q_\sigma)}= CR^{\sigma}$, we can write  for all $y>0$
\[
\int_{E_R}|v^*(x,y)|^2\,dx
\le C R^\sigma \norm{v^*(\cdot,y)}_{L^{q_\sigma}}^2
.
\]
With this information at hand, we can write
\[
\frac1{R^2}\int_0^{2R}\int_{E_R} y^\kappa|v^*|^2\,dxdy
\le
 CR^{\sigma-2}\int_0^{2R}y^\kappa\norm{v^*(\cdot,y)}_{q_\sigma}^2dy .
\]
Then, making the change of variables $y=Rt$, $dy=R\,dt$ and using that $\kappa=1-\sigma$, we get
\begin{align}\label{eq.for.dominated.convergce}
\frac1{R^2}\int_0^{2R}\int_{E_R} y^\kappa|v^*|^2\,dxdy
& \le
C R^{\sigma-2+\kappa+1}\int_0^2t^\kappa\norm{v^*(\cdot,Rt)}_{L^{q_\sigma}}^2 dt
\\ &=
 C \int_0^2t^\kappa\norm{v^*(\cdot,Rt)}_{L^{q_\sigma}}^2 dt
 .
\end{align}
Let $\epsilon>0$. Now considering that $C_c^\infty(\R^3)$ is dense in $L^{q_\sigma}(\R^3)$,  we can choose $v_\epsilon\in C_c^\infty(\R^3)\subset L^1\cap L^{q_\sigma}(\R^3)$ such that $\norm{v-v_\epsilon}_{L^{q_\sigma}}<\varepsilon/2$.
Then, considering Young inequality, we can write
\[
\norm{P_y^{(\sigma)}*v_\epsilon}_{L^{q_\sigma}}\le \norm{P_y^{(\sigma)}}_{L^{q_\sigma} }\,\norm{v_\epsilon}_{L^1}.
\]
Now, note that,  by rescaling we have $P_y^{(\sigma)}(x)=y^{-3}P_1^{(\sigma)}(x/y)$, and then we obtain
\[
\norm{P_y^{(\sigma)}}_{L^{q_\sigma}}=y^{3/q_\sigma-3}\norm{P_1^{(\sigma)}}_{L^{q_\sigma}}=y^{-(3+\sigma)/2}K_\sigma,
\qquad
 K_\sigma:=\norm{P_1^{(\sigma)}}_{L^{q_\sigma}(\R^3)}<+\infty
 .
\]
Hence, for fixed $\epsilon$, we obtain
\[
\norm{P_y^{(\sigma)}*v_\epsilon}_{L^{q_\sigma}}
\le
 K_\sigma
 y^{ -(3+\sigma)/2}\norm{v_\epsilon}_{L^1}\xrightarrow[y\to\infty]{}0
.
\]
With this in hand, we choose $y_0$ such that  $K_\sigma y_0^{-(3+\sigma)/2}\norm{v_\epsilon}_{L^1}<\varepsilon/2$. Then, for each  $y>y_0$,  we use Young's inequality to write
\[
\norm{v^*(\cdot,y)}_{L^{q_\sigma}}
\le
 \norm{P_y^{(\sigma)}*(v-v_\epsilon)}_{L^{q_\sigma}}
 +
 \norm{P_y^{(\sigma)}*v_\epsilon}_{L^{q_\sigma}}
\le
\norm{v-v_\epsilon}_{L^{q_\sigma}}+K_\sigma y^{-(3+\sigma)/2}\norm{v_\epsilon}_{1}<\varepsilon.
\]
This implies that
\begin{equation}\label{eq.infty.y}
\norm{v^*(\cdot,y)}_{L^{q_\sigma}(\R^3)}\xrightarrow[y\to\infty]{}0.
\end{equation}
On the other hand, considering \eqref{eq3.6} and the fact that $\kappa =1-\sigma>-1$ (since $\sigma<2$) we have
\[
t^\kappa \norm{v^*(\cdot,Rt)}_{L^{q_\sigma}}^2
\le
t^\kappa
\norm{v}_{L^{q_\sigma}}^2 \in L^1_t (0,2)
.
\]
Moreover, using \eqref{eq.infty.y}, for all $t\in (0,2)$ we have $\norm{v^*(\cdot,Rt)}_{L^{q_\sigma}}^2 \to 0$ as $R\to \infty$. Then, considering the Dominated Convergence Theorem and \eqref{eq.for.dominated.convergce}, we conclude the claim of the lemma.
\end{proof}

\begin{proof}[Proof of Lemma \ref{lemma I R}]
To begin, recall that $p=\frac{6}{3-\alpha}$ and $q=\frac{6}{3-\beta}$. Then, since
$u\in \dot H^{\frac \alpha 2}(\R^3)$ and $b\in \dot H^{\frac \beta 2}(\R^3)$, by the Sobolev embedding we have $u\in L^{p}(\mathbb R^3)$ and $b\in L^{q}(\mathbb R^3)$, and $3\le p\le \frac 92$ because $1\le\alpha\le\frac 53$.

We first record an elementary scaling estimate. Let $\rho_1,\rho_2,\rho_3\in[1,\infty]$ be such that
\[
\theta:=1-\frac1{\rho_1}-\frac1{\rho_2}-\frac1{\rho_3}\in[0,1).
\]
Then, by H\"older's inequality with exponents $\rho_1,\rho_2,\rho_3,\frac1\theta$, and using that $\|\nabla\psi_R\|_{L^\infty}\le CR^{-1}$ and $|\mathcal C_R|\le CR^3$, we have, for all $R>1$ that 
\begin{multline}\label{eq:scaling}
\int_{\mathcal C_R}|f_1||f_2||f_3||\nabla\psi_R|\,dx
\le
\|f_1\|_{L^{\rho_1}(\mathcal C_R)}\|f_2\|_{L^{\rho_2}(\mathcal C_R)}\|f_3\|_{L^{\rho_3}(\mathcal C_R)}
\|\nabla\psi_R\|_{L^{1/\theta}(\mathcal C_R)}
\\
\le
CR^{3\theta-1}\prod_{i=1}^3\|f_i\|_{L^{\rho_i}(\mathcal C_R)} ,
\end{multline}
where for $\theta=0$ the last norm is understood as an $L^\infty$ norm. In particular, if $3\theta-1\le0$ then $R^{3\theta-1}\le1$ for $R>1$.

Now, to continue, we consider the notation
\begin{equation}
\int_{ \mathcal C_R }\bigl(|\vu|^2+|b|^2\bigr)
|  \vu  |
|  \nabla  \psi_R |
 dx
+
 \int_{ \mathcal C_R }
 | \Pi  |
 | \vu  |
| \nabla \psi_R |
  dx
 =: 
  I_1(R) + I_2(R)+I_3(R)
  ,
\end{equation}
where $I_1(R)$ contains $|u|^3$, $I_2(R)$ contains $|b|^2|u|$, and $I_3(R)$ contains the pressure. With this in hand, we aim to prove
$$
\displaystyle \lim_{R\to +\infty}I_{1}(R)= \lim_{R\to +\infty} I_{2}(R)=\lim_{R\to +\infty} I_{3}(R)=0.
$$

\smallskip
\noindent\textit{The term $I_1$.} We apply \eqref{eq:scaling} with $(\rho_1,\rho_2,\rho_3)=(p,p,p)$, for which $\theta_1=1-\frac3p=\frac{\alpha-1}{2}\in[0,\frac13]$. Then $3\theta_1-1=\frac{3\alpha-5}{2}\le0$, since $\alpha\le\frac53$, and we obtain
\[
I_1(R)\le C\|u\|^3_{L^{p}(\mathcal C_R)} .
\]
Thus, considering that $\lim_{R\to\infty}\|u\|_{L^{p}(\mathcal C_R)}= 0$, we conclude $I_1(R)\to0$ as $R\to+\infty$.

\smallskip
\noindent\textit{The term $I_2$.} We apply \eqref{eq:scaling} with $(\rho_1,\rho_2,\rho_3)=(p,q,q)$, for which
\[
\theta_2=1-\frac1p-\frac2q=1-\frac{3-\alpha}{6}-\frac{3-\beta}{3}=\frac{\alpha+2\beta-3}{6}\ge0,
\]
and $3\theta_2-1=\frac{\alpha+2\beta-5}{2}\le0$, by the assumption $3\le\alpha+2\beta\le5$. Hence
\[
I_2(R)\le C\|u\|_{L^{p}(\mathcal C_R)}\|b\|^2_{L^{q}(\mathcal C_R)}\to0 \quad\text{as } R\to+\infty ,
\]
since $\|b\|_{L^q(\mathcal C_R)}\le \|b\|_{L^q(\R^3)}<+\infty$ and $\|u\|_{L^{p}(\mathcal C_R)}\to0$.

\smallskip
\noindent\textit{The term $I_3$.} We split $\Pi=\Pi_1+\Pi_2$ as in \eqref{eq:pressure}. For $\Pi_1$ we apply \eqref{eq:scaling} with $(\rho_1,\rho_2,\rho_3)=(\frac p2,p,\infty)$ and $f_3\equiv1$, which gives the same $\theta_1=1-\frac2p-\frac1p=\frac{\alpha-1}2$ as before. Note that $\|f_3\|_{L^\infty(\mathcal C_R)}=1$, so no volume factor appears,  the only power of $R$ is $R^{3\theta_1-1}\le1$, coming from $\|\nabla\psi_R\|_{L^{1/\theta_1}(\mathcal C_R)}$. Hence, by \eqref{eq:CZ}, we can write 
\[
\int_{\mathcal C_R}|\Pi_1||u||\nabla\psi_R|\,dx
\le
C R^{3\theta_1-1}\|\Pi_1\|_{L^{\frac p2}(\R^3)}\|u\|_{L^{p}(\mathcal C_R)}
\le
C\|u\|^2_{L^{p}(\R^3)}\|u\|_{L^{p}(\mathcal C_R)}\to0 .
\]
Now, for $\Pi_2$ we apply \eqref{eq:scaling} with $(\rho_1,\rho_2,\rho_3)=(\frac q2,p,\infty)$ and $f_3\equiv1$, so that $\theta=1-\frac2q-\frac1p=\theta_2$ and $R^{3\theta_2-1}\le1$ as before. Thus, by \eqref{eq:CZ}, we get 
\[
\int_{\mathcal C_R}|\Pi_2||u||\nabla\psi_R|\,dx
\le
C R^{3\theta_2-1}\|\Pi_2\|_{L^{\frac q2}(\R^3)}\|u\|_{L^{p}(\mathcal C_R)}
\le
C\|b\|^2_{L^{q}(\R^3)}\|u\|_{L^{p}(\mathcal C_R)}\to0 .
\]
Hence, we conclude that $I_3(R)\to0$ as $R\to+\infty$.
Then, by gathering together the limits for $I_{1}(R)$, $I_{2}(R)$ and $I_3(R)$, we deduce the desired limit.
\end{proof}

\section{Proof of Theorem \ref{thm.radial}}\label{sect:radial}

We begin with the lemma that will be used in the proof of Theorem \ref{thm.radial}, whose proof is presented in Subsection \ref{sub:radial}. The other tool we need, namely Lemma \ref{lemma.with.E}, was already stated and proved in Section \ref{proofs} for a generic exponent.

\begin{lemma}\label{lemma I R radial}
Assume the hypotheses of Theorem \ref{thm.radial}. Then
\begin{multline*}
\lim_{R\to +\infty}
\Big(\int_{ \mathcal C_R }\big(|u|^2+|b|^2\big)
| u_\rho |
| \partial_\rho\psi_R |
 dx
 \\
 +
 \int_{ \mathcal C_R }
 \big(| \Pi_u  |+| \Pi_b  |\big)
 | u_\rho |
| \partial_\rho\psi_R |
  dx
+
 \int_{ \mathcal C_R }
 | u|\,| b  |
 | b_\rho |
| \partial_\rho\psi_R |
  dx\Big)
 = 0.
\end{multline*}
\end{lemma}

\begin{proof}[Proof of Theorem \ref{thm.radial}]
The proof follows the lines of that of Theorem \ref{thm1}, and we only indicate the changes. We use the pressure $\Pi=\Pi_1+\Pi_2$ of \eqref{eq:pressure} and the cut-offs $\psi_R$, $\chi_R$ and $\Phi_R$ applied there. In this section we also write $\Pi_u:=\Pi_1$ and $\Pi_b:=-\Pi_2$, so that $\Pi=\Pi_u-\Pi_b$ and, by \eqref{eq:CZ},
\[
\|\Pi_u\|_{L^{\frac{p_\alpha}{2}}(\mathbb{R}^3)}\le C\,\|u\|_{L^{p_\alpha}(\mathbb{R}^3)}^2<+\infty,
\qquad
\|\Pi_b\|_{L^{\frac{p_\beta}{2}}(\mathbb{R}^3)}\le C\,\|b\|_{L^{p_\beta}(\mathbb{R}^3)}^2<+\infty .
\]
Since $\psi_R$ is radial, we have $\nabla\psi_R(x)=\partial_\rho\psi_R(x)\frac{x}{|x|}$, with $|\partial_\rho\psi_R|=|\nabla\psi_R|\le CR^{-1}$. Hence, we can write 
\[
u\cdot\nabla\psi_R=u_\rho\,\partial_\rho\psi_R,\qquad b\cdot\nabla\psi_R=b_\rho\,\partial_\rho\psi_R .
\]
Testing the equations with $\psi_Ru$ and $\psi_Rb$, using the Caffarelli--Silvestre extension and arguing exactly as in the proof of Theorem \ref{thm1}, we arrive at \eqref{eq3-energy}, and the two terms on the right-hand side of \eqref{eq3-energy} involving $u^*$ and $b^*$ tend to zero as $R\to\infty$, by Lemma \ref{lemma.with.E}. The only difference is the treatment of the flux term, which now we write using $u\cdot\nabla\psi_R=u_\rho\,\partial_\rho\psi_R$, $b\cdot\nabla\psi_R=b_\rho\,\partial_\rho\psi_R$, $\Pi=\Pi_u-\Pi_b$ and $|u\cdot b|\le|u||b|$:
\begin{multline}\label{FR}
\Big|\int_{\mathcal C_R}\Big(\frac{|u|^2}{2}+\frac{|b|^2}{2}+\Pi\Big)u_\rho\,\partial_\rho\psi_R\,dx-\int_{\mathcal C_R}(u\cdot b)\,b_\rho\,\partial_\rho\psi_R\,dx\Big|
\\
\leq
C
\int_{ \mathcal C_R }
\big(|u|^2+|b|^2\big)
| u_\rho |
| \partial_\rho\psi_R |
 dx
+
C
 \int_{ \mathcal C_R }
 \big(| \Pi_u  |+| \Pi_b  |\big)
 | u_\rho |
| \partial_\rho\psi_R |
  dx
+
C
 \int_{ \mathcal C_R }
 | u|\,| b  |
 | b_\rho |
| \partial_\rho\psi_R |
  dx .
\end{multline}
By Lemma \ref{lemma I R radial}, the right-hand side tends to zero as $R\to\infty$. Thus, taking the limit as $R\to\infty$ in \eqref{eq3-energy}, and using the monotone convergence theorem as in the proof of Theorem \ref{thm1}, we get
\[
C_\alpha\int_{\mathbb{R}^4_+} y^\mu |\overline{\nabla} u^*|^2 \, dx \, dy
+
C_\beta\int_{\mathbb{R}^4_+} y^\nu |\overline{\nabla} b^*|^2 \, dx \, dy = 0.
\]
Hence $u$ and $b$ are constant, and then, since $u\in L^{p_\alpha}(\mathbb R^3)$ and $b\in L^{p_\beta}(\mathbb R^3)$ by the Sobolev embedding, and the only constant function belonging to $L^{p}(\mathbb R^3)$, $p<\infty$, is the zero function, we conclude that $u\equiv 0$ and $b\equiv0$. This finishes the proof.
\end{proof}

\subsection{Proof of Lemma \ref{lemma I R radial}}\label{sub:radial}

To begin, we stress the fact that each term in the proof below is estimated by H\"older's inequality.
This requires an exponent $\sigma$ for $u_\rho$ or $b_\rho$ such that
\[
0\le \tau-\frac1\sigma\le\frac13
\]
for a suitable $\tau$. The following elementary lemma gives a
criterion for the existence of such an exponent.

\begin{lemma}\label{lemma.exponent}
Let $\tau\in(0,1)$, $p\in(1,\infty)$, $r\in[1,\infty]$, and
\[
g\in L^p(\mathbb R^3)\cap L^r(\mathbb R^3).
\]
Assume that
\begin{equation}\label{E}
\frac1r\le\tau
\quad\text{if}\quad \frac1p>\tau,
\qquad
\frac1r\ge\tau-\frac13
\quad\text{if}\quad
\frac1p<\tau-\frac13.
\end{equation}
Then there exists $\sigma\in(1,\infty)$ such that
\[
g\in L^\sigma(\mathbb R^3),
\qquad
0\le\tau-\frac1\sigma\le\frac13.
\]
\end{lemma}

\begin{proof}
By interpolation, $g\in L^\sigma(\mathbb R^3)$ whenever
$1/\sigma$ lies between $1/p$ and $1/r$. We choose $1/\sigma$
by projecting $1/p$ onto the interval
$[\tau-\frac13,\tau]$, more precisely 
\begin{equation}\label{sigma.general}
\frac1\sigma
=
\min\left\{
\max\left\{\frac1p,\tau-\frac13\right\},\tau
\right\}.
\end{equation}
In the following we distinguish three cases. If
\[
\tau-\frac13\le\frac1p\le\tau,
\]
we take $\sigma=p$, and the conclusion follows.
If $1/p>\tau$, then $1/\sigma=\tau$. The first condition in
\eqref{E} gives $1/r\le\tau<1/p$, thus $1/\sigma$ lies between
$1/r$ and $1/p$, and then $g\in L^\sigma(\mathbb R^3)$. 
Finally, if $1/p<\tau-\frac13$, then
$1/\sigma=\tau-\frac13$. The second condition in \eqref{E} gives
\[
\frac1p<\tau
-
\frac13
\le
\frac1r
,
\]
 and thus, considering again interpolation, we get  $g\in L^\sigma(\mathbb R^3)$.
 In particular, note that in all three cases we have 
\[
0<\frac1\sigma\le\tau<1.
\]
Thus, we conclude that  $1<\sigma<\infty$, and
$0\le\tau-\frac1\sigma\le\frac13$. This finishes the proof.
\end{proof}
In the next, we apply this lemma with $|u_\rho|\le |u|$ and
$|b_\rho|\le |b|$. At this point is convenient to  recall that
\[
\frac1{p_\alpha}=\frac12-\frac\alpha6,
\qquad
\frac1{p_\beta}=\frac12-\frac\beta6.
\]
We use the lemma with
\[
(g,\tau,p,r)
=
\left(u_\rho,\frac{\alpha}{3},p_\alpha,r_1\right),
\quad
\left(u_\rho,\frac{\beta}{3},p_\alpha,r_1\right),
\quad
\left(b_\rho,\frac{\alpha+\beta}{6},p_\beta,r_2\right).
\]
For clarity of the exposition, the resulting exponents can be written explicitly as
\begin{equation}\label{sigma1}
\frac1{\sigma_1}
=
\min\left\{
\max\left\{
\frac12-\frac\alpha6,\frac{\alpha-1}{3}
\right\},
\frac{\alpha}{3}
\right\},
\end{equation}
\begin{equation}\label{sigma2}
\frac1{\sigma_2}
=
\min\left\{
\max\left\{
\frac12-\frac\alpha6,\frac{\beta-1}{3}
\right\},
\frac{\beta}{3}
\right\},
\end{equation}
and
\begin{equation}\label{sigma3}
\frac1{\sigma_3}
=
\min\left\{
\max\left\{
\frac12-\frac\beta6,\frac{\alpha+\beta-2}{6}
\right\},
\frac{\alpha+\beta}{6}
\right\}.
\end{equation}
At this point is convenient to  recall that
$\frac1{p_\alpha}=\frac12-\frac\alpha6,
$ and 
$\frac1{p_\beta}=\frac12-\frac\beta6$.
With this,  note that 
\begin{alignat*}{2}
\frac1{p_\alpha}>\frac{\alpha}{3}
&\iff \alpha<1,
&\qquad
\frac1{p_\alpha}<\frac{\alpha}{3}-\frac13
&\iff \alpha>\frac53,\\
\frac1{p_\alpha}>\frac{\beta}{3}
&\iff 
\alpha+2\beta<3,
&
\qquad
\frac1{p_\alpha}<\frac{\beta}{3}-\frac13
&\iff 
\alpha+2\beta>5,\\
\frac1{p_\beta}>\frac{\alpha+\beta}{6}
&\iff
 \alpha+2\beta<3,
&\qquad
\frac1{p_\beta}<\frac{\alpha
+\beta}{6}-\frac13
& 
\iff
 \alpha+2\beta 
 >
 5.
\end{alignat*}
As a consequence of this, the conditions \eqref{E} in the three applications above
are exactly the assumptions (R1), (R2), and (R3), respectively.

Hence, under the hypotheses of Theorem \ref{thm.radial}, we have
$
u_\rho\in L^{\sigma_1}(\mathbb R^3)
\cap L^{\sigma_2}(\mathbb R^3)
$, and 
$
b_\rho\in L^{\sigma_3}(\mathbb R^3)$, where
\begin{equation}\label{thetas}
\theta_1:=\frac{\alpha}{3}-\frac1{\sigma_1},
\qquad
\theta_2:=\frac{\beta}{3}-\frac1{\sigma_2},
\qquad
\theta_3:=\frac{\alpha+\beta}{6}-\frac1{\sigma_3}
\end{equation}
all belong to $[0,\frac13]$.  

\begin{proof}[Proof of Lemma \ref{lemma I R radial}]
To begin, recall that $p_\alpha=\frac{6}{3-\alpha}$ and $p_\beta=\frac{6}{3-\beta}$. Then, since $u\in\dot H^{\frac\alpha2}(\R^3)$ and $b\in\dot H^{\frac\beta2}(\R^3)$, by the Sobolev embedding we have $u\in L^{p_\alpha}(\R^3)$ and $b\in L^{p_\beta}(\R^3)$, and by the Calder\'on--Zygmund estimate \eqref{eq:CZ} we have $\Pi_u\in L^{\frac{p_\alpha}{2}}(\R^3)$ and $\Pi_b\in L^{\frac{p_\beta}{2}}(\R^3)$. Moreover, let $\sigma_1,\sigma_2,\sigma_3\in[1,\infty)$ be the exponents constructed in the previous subsection, so that
\[
u_\rho\in L^{\sigma_1}(\R^3)\cap L^{\sigma_2}(\R^3),\qquad b_\rho\in L^{\sigma_3}(\R^3),
\]
and the numbers $\theta_1,\theta_2,\theta_3$ defined in \eqref{thetas} belong to $[0,\frac13]$.

We first record an elementary estimate. Let $\varrho_1,\varrho_2,\varrho_3\in[1,\infty]$ be such that
\[
\theta:=1-\frac1{\varrho_1}-\frac1{\varrho_2}-\frac1{\varrho_3}\in\Bigl[0,\frac13\Bigr].
\]
Then, by H\"older's inequality with exponents $\varrho_1,\varrho_2,\varrho_3,\frac1\theta$, and using that $\|\partial_\rho\psi_R\|_{L^\infty}\le CR^{-1}$ and $|\mathcal C_R|\le CR^3$, we have, for all $R>1$,
\begin{multline}\label{eq:scaling.radial}
\int_{\mathcal C_R}|f_1||f_2||f_3||\partial_\rho\psi_R|\,dx
\le
\|f_1\|_{L^{\varrho_1}(\mathcal C_R)}\|f_2\|_{L^{\varrho_2}(\mathcal C_R)}\|f_3\|_{L^{\varrho_3}(\mathcal C_R)}
\|\partial_\rho\psi_R\|_{L^{1/\theta}(\mathcal C_R)}
\\
\le
CR^{3\theta-1}\prod_{i=1}^3\|f_i\|_{L^{\varrho_i}(\mathcal C_R)}
\le
C\prod_{i=1}^3\|f_i\|_{L^{\varrho_i}(\mathcal C_R)},
\end{multline}
where for $\theta=0$ the last factor in the first line is understood as an $L^\infty$ norm, and the last inequality holds because $3\theta-1\le0$ and $R>1$.

Now, to continue, we consider  
\begin{equation}
\begin{aligned}
&\int_{\mathcal C_R}\bigl(|u|^2+|b|^2\bigr)|u_\rho||\partial_\rho\psi_R|\,dx
+\int_{\mathcal C_R}\bigl(|\Pi_u|+|\Pi_b|\bigr)|u_\rho||\partial_\rho\psi_R|\,dx
+\int_{\mathcal C_R}|u||b||b_\rho||\partial_\rho\psi_R|\,dx\\
&\qquad=:I_1(R)+I_2(R)+I_3(R)+I_4(R)+I_5(R),
\end{aligned}
\end{equation}
where $I_1$ contains $|u|^2|u_\rho|$, $I_2$ contains $|b|^2|u_\rho|$, $I_3$ and $I_4$ contain $|\Pi_u||u_\rho|$ and $|\Pi_b||u_\rho|$ respectively, and $I_5$ contains $|u||b||b_\rho|$. With this in hand, we aim to prove that $\lim_{R\to+\infty}I_k(R)=0$ for $k=1,\dots,5$.

\smallskip
\noindent\textit{The term $I_1$.} Since $|u_\rho|\le|u|$ and $\frac1{p_\alpha}=\frac12-\frac\alpha6$, we apply \eqref{eq:scaling.radial} with $(\varrho_1,\varrho_2,\varrho_3)=(p_\alpha,p_\alpha,\sigma_1)$, for which
\[
\theta=1-\frac{2}{p_\alpha}-\frac1{\sigma_1}=\frac{\alpha}{3}-\frac1{\sigma_1}=\theta_1\in\Bigl[0,\frac13\Bigr].
\]
Hence
\[
I_1(R)\le C\|u\|^2_{L^{p_\alpha}(\R^3)}\|u_\rho\|_{L^{\sigma_1}(\mathcal C_R)}.
\]

\smallskip
\noindent\textit{The term $I_2$.} We apply \eqref{eq:scaling.radial} with $(\varrho_1,\varrho_2,\varrho_3)=(p_\beta,p_\beta,\sigma_2)$, for which $\theta=1-\frac{2}{p_\beta}-\frac1{\sigma_2}=\frac{\beta}{3}-\frac1{\sigma_2}=\theta_2\in[0,\frac13]$. Hence
\[
I_2(R)\le C\|b\|^2_{L^{p_\beta}(\R^3)}\|u_\rho\|_{L^{\sigma_2}(\mathcal C_R)}.
\]

\smallskip
\noindent\textit{The terms $I_3$ and $I_4$.} We apply \eqref{eq:scaling.radial} with $f_3\equiv1$ and $\varrho_3=\infty$, so that no volume factor appears and the only power of $R$ is the one in \eqref{eq:scaling.radial}. For $I_3$ we take $(\varrho_1,\varrho_2,\varrho_3)=(\frac{p_\alpha}{2},\sigma_1,\infty)$, which gives $\theta=1-\frac2{p_\alpha}-\frac1{\sigma_1}=\theta_1$, and for $I_4$ we take $(\varrho_1,\varrho_2,\varrho_3)=(\frac{p_\beta}{2},\sigma_2,\infty)$, which gives $\theta=\theta_2$. Then, by \eqref{eq:CZ},
\[
I_3(R)\le C\|\Pi_u\|_{L^{\frac{p_\alpha}{2}}(\R^3)}\|u_\rho\|_{L^{\sigma_1}(\mathcal C_R)}\le C\|u\|^2_{L^{p_\alpha}(\R^3)}\|u_\rho\|_{L^{\sigma_1}(\mathcal C_R)},
\]
\[
I_4(R)\le C\|\Pi_b\|_{L^{\frac{p_\beta}{2}}(\R^3)}\|u_\rho\|_{L^{\sigma_2}(\mathcal C_R)}\le C\|b\|^2_{L^{p_\beta}(\R^3)}\|u_\rho\|_{L^{\sigma_2}(\mathcal C_R)}.
\]

\smallskip
\noindent\textit{The term $I_5$.} We apply \eqref{eq:scaling.radial} with $(\varrho_1,\varrho_2,\varrho_3)=(p_\alpha,p_\beta,\sigma_3)$, for which $\theta=1-\frac1{p_\alpha}-\frac1{p_\beta}-\frac1{\sigma_3}=\frac{\alpha+\beta}{6}-\frac1{\sigma_3}=\theta_3\in[0,\frac13]$. Hence
\[
I_5(R)\le C\|u\|_{L^{p_\alpha}(\R^3)}\|b\|_{L^{p_\beta}(\R^3)}\|b_\rho\|_{L^{\sigma_3}(\mathcal C_R)}.
\]

Finally, since $u_\rho\in L^{\sigma_1}\cap L^{\sigma_2}$ and $b_\rho\in L^{\sigma_3}$ with $\sigma_1,\sigma_2,\sigma_3<\infty$, we have $\|u_\rho\|_{L^{\sigma_i}(\mathcal C_R)}\to0$ for $i=1,2$ and $\|b_\rho\|_{L^{\sigma_3}(\mathcal C_R)}\to0$ as $R\to+\infty$, being tails of convergent integrals. Since $\|u\|_{L^{p_\alpha}}$ and $\|b\|_{L^{p_\beta}}$ are finite, we get $\lim_{R\to+\infty}I_k(R)=0$ for $k=1,\dots,5$. Then, by gathering together these limits, we deduce the desired limit.
\end{proof}
 
 \medskip

\paragraph{\bf Datasets} Data sharing does not apply to this article as no datasets were generated or analyzed during the
current study.

\paragraph{\bf  Conflict of interest} In addition, the authors declare no conflict of interest and confirm that they contributed equally to this paper.


 \end{document}